\documentclass[12pt]{amsart}
\usepackage[T1]{fontenc}
\usepackage[pdftex]{color,graphicx}
\usepackage{bbm} 
\usepackage{amsmath}
\usepackage{mathabx}
\usepackage[normalem]{ulem}
\usepackage{dsfont}
\usepackage{url}
\usepackage{mathtools}
\mathtoolsset{showonlyrefs=false}
\usepackage{anysize}
\usepackage{nicefrac}
\usepackage{geometry}
\marginsize{2.5cm}{2.5cm}{2.5cm}{2.5cm}
\usepackage{amssymb}
\usepackage{hyperref}
\usepackage{color}
\usepackage{orcidlink}
\usepackage[dvipsnames]{xcolor}

\usepackage{soul}
\theoremstyle{plain}
\numberwithin{equation}{section}
\newtheorem{theorem}{Theorem}[section]
\newtheorem{lemma}[theorem]{Lemma}

\newtheorem{remark}[theorem]{Remark}

\renewcommand{\Pr}{\mathbb{P}}

\newcommand{\PP}[1]{\left(#1\right)}
\newcommand{\CC}[1]{\left[#1\right]}
\newcommand{\LL}[1]{\left\{#1\right\}}
\newcommand{\R}{\mathds{R}}
\newcommand{\N}{\mathds{N}}

\newcommand{\C}{\mathds{C}}
\newcommand{\dd}{\textnormal{d}}

\newcommand{\E}[2][ ]{\mathbb{E}{#1}\left[#2\right]}

\newcommand{\abs}[1]{\left|#1\right|}
\newcommand{\norm}[1]{\left|\left|#1\right|\right|}

\renewcommand{\Pr}[1]{\mathds{P}\left(#1\right)}
\renewcommand{\line}{\overline}

\title[Random banded Toeplitz matrix and random polynomial]{The condition number of a random banded Toeplitz matrix is typically large}

\author{Paulo~Manrique\orcidlink{00000-0002-4908-5514}}
\address{
\url{https://orcid.org/0000-0002-4908-5514}}
\email{pmanriquem@ipn.mx}



\begin{document}

\begin{abstract}
It is well known that square matrices with independent and identically distributed (iid) random entries are typically well conditioned. A natural question is whether this favorable behavior persists for random matrices whose entries obey additional structure, i.e., their position inside of the matrix. A prominent class of structured matrices is given by {\it Toeplitz matrices}, characterized by constant diagonals. A particular tractable subclass is that of circulant matrices, whose additional characteristic (its entries {\it circulate} row by row) allows one to express their conditioning in terms of the localization of the zeros of a associated polynomial. When the entries of a circulant matrix are iid, the matrix is well conditioned precisely when the corresponding random polynomial has no zeros on the unit circle. This connection is especially relevant because, as the degree of a random polynomial increases, its zeros tend to concentrate near the unit circle, making it a delicate problem to quantify how close the closest zeros lie to the unit circle. Another notable family within the Toeplitz class is that of {\it banded Toeplitz matrices}, namely matrices for which only finitely many diagonals around the main diagonal may be nonzero. These matrices have been extensively studied in Operator Theory, nevertheless, despite their apparent simplicity, they raise subtle questions regarding the behavior of their condition numbers. In present work we show that the bandwidth asymmetry plays a decisive role: if the band contains $r$ diagonals below and $s$ diagonal above the main diagonal, then if $r=s$ the banded Toeplitz matriz is well conditioned with high probability, whereas if $r\neq s$ it is typically ill conditioned. This highlights that structural constraints can have a impact on the numerical behavior of random matrices.

\end{abstract}
\maketitle

\section{Introduction}

Toeplitz matrices are structured matrices of considerable scientific importance, with numerous applications in engineering. They have been extensively studied within operator theory, where one of the central topics of interest is the asymptotic behavior of their spectra as the dimension grows large \cite{bottcher2005spectral, bottcher2006analysis}. From a numerical perspective, their structure arises naturally in the discretization of certain differential equations and in time series analysis \cite{chan2007introduction}. In these contexts, spectral properties again play a crucial role, particularly in the study of the condition number of a matrix, which governs numerical stability and is determined by its extreme singular values \cite{ng2004iterative, chan2007introduction}.

Here we are interested in understanding the {\it stability} of banded Toeplitz matrices. To make the presentation precise, we recall that the singular values of a (square or rectangular) matrix $A$ are the eigenvalues of the the matrix $\sqrt{A^T A}$, where $A^T$ denotes the transpose of $A$. In particular, all singular values are non--negative. The extreme singular values are closely related to the operator norm of a matrix. If $A$ es a square matrix of size $n\times n$, it operator norm is defined by 
\[
\norm{A} := \max_{\norm{x}_2=1} \norm{Ax}_2,
\]
where $\norm{\cdot}_2$ denotes the Euclidean norm. Denoting the singular values of $A$ by
\[
0\leq \sigma_{n} \leq \sigma_{n-1}\leq \cdots \leq \sigma_{1},
\]
we have
\[
\norm{A} = \sigma_1.
\]
Note that $A$ is singular if and only if $\sigma_{n}=0$. If $A$ is non-singular, then
\[
\norm{A^{-1}} = \PP{\min_{\norm{x}_2=1} \norm{Ax}_2}^{-1} = \sigma_{n}^{-1}.
\]
This identity shows that $\norm{A^{-1}}$ increasing as $\sigma_{n}$ approaches zero, indicating $A$ is {\it nearly singular}. In fact, $\sigma_n$ measure the distance of $A$ to the set of singular matrices \cite{bottcher2005spectral, burgisser2013condition, golub2013matrix}. The {\it condition number} $\kappa(A)$ of $A$ is defined by 
\[
\kappa(A) := \frac{\sigma_1}{\sigma_n},
\]
and quantifies the sensitivity of the the linear system $Ax=b$ to perturbation in $A$ or in $b$, in particular, large values of $\kappa(A)$ indicate potential loss of numerical accuracy \cite{burgisser2013condition, goldstine1951numerical, smale1985efficiency,wozniakowski1976numerical}.

The present work is concerned with understanding the condition number of banded Toeplitz matrices. To this end, we first  recall the definition of a Toeplitz matrix, which is characterized by constant diagonals. More precisely, a matrix $T_n$ of dimension $n\times n$ is said to be Toeplitz if it has the following form:
\[
T_n :=
\CC{\begin{array}{ccccc}
x_0 & x_1 & \ldots & x_{n-2} & x_{n-1} \\
x_{-1} & x_{0} & \ldots & x_{n-3} & x_{n-2} \\
x_{-2} & x_{-1} & \ldots & x_{n-4} & x_{n-3} \\
& & \ldots & & \\
x_{-n+1} & x_{-n+2} & \ldots & x_{-1} & x_{0} 
\end{array}}. 
\]
A banded Toeplitz matrix is a Toeplitz matrix for which $x_j=0$ whenever $j<-r$ or $j>s$, for some fixed positive integers $r,s$.

In general, understanding the condition number of a Toeplitz matrix is a challenging problem \cite{PhysRevE.102.040101, bogomolny2021statistical}. Nevertheless, under various assumptions on the nature of the coefficients $x_{j}$, meaningful results can be obtained (see \cite{bogomolny2021statistical, manrique-miron2022random, onatski2026smallest} and the references therein. In the case of banded Toeplitz matrices, we refer to the classic reference \cite{bottcher2005spectral} for a detailed {\it deterministic} study of their singular values and condition number. In contrast, in the present work is concerned with analyzing the behavior of banded Toeplitz matrices when their entries are random variables, i.e., when the matrix $T_n$ takes the form:
\[
T_n= T_n(r,s) :=
\CC{\begin{array}{ccccc}
\xi_{0} & \ldots & \xi_{s} &  &  \\
\vdots & \ddots &  & \ddots& \\
\xi_{-r} &  & \ddots & & \xi_{s} \\
 & \ddots &  & \ddots & \vdots \\
 & &  & \xi_{-r} & \xi_{0} \\
\end{array}}
\]
where $\xi_j$ are random variables. We say that $T_n(r,s)$ has a bandwidth $r+s+1$. 

Banded Toeplitz matrices $T_n(r,s)$ can be associated with so--called Laurent polynomials $P$, which are expressions of the form
\[
P(z) =P_{r,s}(z) := \sum_{j=-r}^{s} \xi_j z^{j},
\]
where $z\in\C$. The function $P$ is commonly referred to as the {\it symbol} of the banded Toeplitz matriz $T_n$. It is important to note that if the bandwidth of $T_n(r,s)$ remains fixed, i.e., $r,s$ are fixed integers, then the asymptotic behavior of the sequence of matrices $\LL{T_n(r,s): n=1,2,\ldots}$ can be analyzed through the Laurent Polynomial $P_{r,s}$, which does not depend on $n$ \cite{bottcher2005spectral}. 

A standard polynomial can be associated with a Laurent polynomial. Let $Q(z)=Q_{r+s}(z)=z^rP_{r,s}(z)$, where $P_{r,s}(z)$ is a Laurent polynomial. Obverse that $Q_{r+s}$ is a polynomial of degree $r+s$. The nature of the roots of $Q$ describes the behavior of the condition number of the associated banded Toeplitz matrix \cite{amodio1996conditioning}. To state this, we introduce the following notation. We say that $Q_{r,s}$ is of type $(m,u,l)$ if 
\begin{itemize}
\item it has $m$ zeros with modulus strictly less than $1$,
\item $u$ zeros with modulus equal to $1$,
\item and $l$ zeros with modulus strictly greater that $1$.
\end{itemize}
Then the following holds (see Theorem 3 in \cite{amodio1996conditioning}):
\begin{itemize}
\item If $Q_{r+s}$ is of type $(r,0,s)$, there exists a positive constant $C$ such that \[ \kappa(T_n)\leq C \quad \mbox{ for all $n$}.\]
\item If $Q_{r+s}$ is of type $(r_1,r_2,s)$ or $(r,s_1,s_2)$, where $r_1+r_2=r$ and $s_1+s_2=s$, then \[ \kappa(T_n) = \textnormal{O}(n^{v}), \] where $v$ denotes the maximum multiplicity among the zeros of modulus one.
\end{itemize} 

If the banded Toeplitz matrix is either upper or lower triangular, the above results are both necessary and sufficient, see Theorem 1 and Theorem 2  in \cite{amodio1996conditioning}).

When a random banded Toeplitz matrix is considered, the associated Laurent polynomial is also random, and consequently the polynomial $Q$ is random as well. In \cite{barrera2022zero} (Theorem 1.1), it is shown that when the coefficients $\xi_j$ are independent and identically distributed random variables satisfying $\E{\xi_0}=0$, $\E{\xi_0^2}<+\infty$, and
\[
\sup_{u\in\R} \Pr{\abs{\xi_0 - u}\leq \gamma} \leq 1-q,\quad and\quad \Pr{\abs{\xi_0}>M} \leq \frac{q}{2}
\]
for some constants $M>0$, $\gamma>0$, and $q\in(0,1)$, the following holds. If $m=r+s$, then for every $t\geq 1$ one has
\[
\Pr{\min_{\abs{\abs{z}-1} \leq t m^{-2}(\log m)^{-3}} \abs{Q_{m}} \leq t m^{-1/2}(\log m)^{-2} } = \textnormal{O}\PP{(\log m)^{-1/2}},
\]
where the implicit constant depends on $t$ and the distribution of $\xi_0$.

When $m=r+s$ is sufficiently large, with high probability the polynomial $Q_m$ has no zeros on the unit circle. The study of the zeros of random polynomials is a well--established area of research, see \cite{barrera2022zero, bharucha-reid2014random, cook2023universality, ibragimov2013distribution, michelen2020closest, yakir2021approximately} and the references therein. In fact, it is shown in \cite{ibragimov2013distribution} that the condition $\E{\log(1+\abs{\xi_0})} < +\infty$ holds if and only if the zeros of $Q_m$ are asymptotically ($m\to\infty$) concentrated around the unit circle. Moreover, under suitable assumptions, it is shown in \cite{cook2023universality, michelen2020closest} that the zero closest to the unit circle lies at a distance approximately $m^{-2}$.

As we see in the previous comments, in order to understand the behavior of $\kappa(T_n)$ when $T_n$ is a banded Toeplitz matrix, one must determine how many zeros of the associated symbol lie inside the unit disk. This idea is made precise by the following observation, which relies on the notion of the winding number. Let \[ \mathbf{T}:=\LL{t\in\C:\abs{t}=1}. \] If $a:\mathbf{T} \to \C\setminus {0}$ is continuous, then as $t$ moves on $\mathbf{T}$ counterclockwise, the image $a(t)$ traces a continuous closed curve in $\C\setminus{0}$. The number of times this curve winds around the origin in the counterclockwise direction is called the winding number, denoted by $\textnormal{wind}(a)$.

Now consider a Laurent polynomial $P(z)= \sum_{j=-r}^{s} \xi_j z^j$ such that $P(z)\neq 0$ for all $z\in\mathbf{T}$. Then 
\[
\textnormal{wind}(P) = \nu - r,
\]
where $\nu$ denotes the number of zeros of $P$ in the open disk $\mathds{D}:=\LL{z\in\C : \abs{z} < 1}$ (see expression 1.2 in \cite{bottcher2005spectral}). Observe that
\[
P(z) = \sum_{j=-r}^{s} \xi_j z^j = z^{-r} \sum_{j=0}^{r+s} \xi_j z^j,
\]
so $P$ has a zero on $\mathbf{T}$ if and only if $G_{r+s}(z) := \sum_{j=0}^{r+s} \xi_j z^j $ has a zero on $\mathbf{T}$. By the discussion above, when $r+s$ is large, $G_{r+s}$ does not vanish on $\mathbf{T}$ with high probability.

In Chapter 3 of \cite{bottcher2005spectral}, a sequence of matrices ${A_n}$, where each $A_n$ is of size $n\times n$, is said to be {\it stable} if there exist $n_0\leq 1$ and $M\in(0,\infty)$  such that $A_n$ is invertible for all $n\geq n_0$ and
\[
\norm{A_n^{-1}} \leq M \quad \mbox{ for all }\quad  n\geq n_0.
\]

Since a sequence ${T_n}$ of banded Toeplitz matrices with the same band can be associated with the same Laurent polynomial $P(z)$, it shown in Chapter 3 of \cite{bottcher2005spectral} that ${T_n}$ is stable if and only if $P$ has no zeros on $\mathbf{T}$ and $\textnormal{wind}(T)=0$. This criterion is consistent with the classification in \cite{amodio1996conditioning} when the banded Toeplitz matrix is of type $(r,0,s)$. Accordingly, given a Laurent polynomial $P_{r,s}$, we are led to ask for the number of zeros of the associated polynomial $G_{r+s}$ that lie inside the unit disk. This question has been studied previously in \cite{yakir2021approximately} in the case of {\it Littlewood polynomials}, namely polynomials with coefficients in ${-1,+1}$. Littlewood polynomials can be viewed as random polynomials of the form
\[
G_m(z) = \sum_{j=0}^{m} \xi_j z^j,
\]
where the coefficients $\xi_j$ are independent and identically distributed (iid) Rademacher random variables, i.e.,
\[
\Pr{\xi_0 = 1} = \Pr{\xi_0 = -1} = \frac{1}{2}.
\]

We denote by  
\[
\nu_n := \sum_{z:G_n(z)} \delta_z
\]
the random counting measure of the zeros of $G_n$, where $\delta_z$ denotes the Dirac mass at $z$. In \cite{yakir2021approximately}, it is proved that 
\begin{equation}\label{eqn:main.expression}
\lim_{m\to\infty} \Pr{\abs{\nu_m(\mathds{D}) - \frac{m}{2}} \geq m^{9/10}} = 0.
\end{equation}

It follows that approximately one half of the zeros of a Littlewood polynomial line in $\mathds{D}$. Consequently, if we consider a banded Toeplitz matrix whose entries are iid Rademacher random variables, then the number of zeros in $\mathds{D}$ of the associated Laurent polynomial $P(z)=z^{-r} G_{r+s}(z)$ is, with high probability, approximately
\[
\frac{r+s}{2}
\]
provided that $r+s$ is sufficiently large. If moreover $\frac{r+s}{2} = r$, which occurs when $r=s$, then the associated banded Toeplitz matrix has symmetric bandwidth about the main diagonal. In addition, with high probability $P$ has no zeros on $\mathbf{T}$, which suggests that the banded Toeplitz matriz is well--conditioned. In other words, with high probability the sequence of banded Toeplitz matrices ${T_n}$ associated with the same Laurent polynomial satisfies 
\[
\kappa(T_n) \leq M \quad \mbox{ for all sufficiently large $n$},
\]
for some constant $M\in(0,\infty)$. It is important to emphasize that this conclusion would hold only in the symmetric band case, $r=s$. Indeed, if for instance $m=r+s$ is large and $r=\lambda m$ and $s=(1-\lambda)m$ for $\lambda\in(0,1)$ with $\lambda \neq \frac{1}{2}$, then
\[
\frac{m}{2} = \frac{r+s}{2} \neq r = \alpha m.
\]
In this regime, with high probability the winding number satisfies 
\[
\textnormal{wind}(P) \approx \frac{m}{2} - \alpha m \neq 0.
\]
When the winding number of a Laurent polynomial with no zeros on $\mathbf{T}$ is nonzero, one observe the following phenomenon. Define
\[
\delta := \max\LL{\abs{\delta_j}<1 : G_{r+s}(\delta_j)=0 }, \quad \mu := \min\LL{\abs{\mu_k}>1 : G_{r+s}(\mu_k)=0 }.
\]
Then (see Theorem 4.1 in \cite{bottcher2005spectral}), for every
\[
\alpha < \min\LL{\log\frac{1}{\delta},\log \mu}
\]
there exists a constant $C_{\alpha}>0$, depending only on $\alpha$ and on $P$, such that
\[
\norm{T_n^{-1}}_2 \geq C_{\alpha} e^{\alpha n}\quad \mbox{ for all $n\geq 1$},
\] 
where $T_n$ denotes the banded Toeplitz matrix associated with $P(z)$. This implies that banded Toeplitz matrices with iid Rademacher entries, when $r+s$ is large and the bandwidth parameters $r,s$ are not {\it balanced} (i.e., $r\neq \frac{r+s}{2}$), tend to be poorly conditioned with high probability.

The present work aims to understand the order of the probability of the event appearing in \ref{eqn:main.expression} for a broad class of random coefficients, and to infer from this the behavior of the condition number of a random banded Toeplitz matrix. Let us emphasize that the strategy we follow to estimate the quantity in \ref{eqn:main.expression} is, in essence, based on the same steps as in \cite{yakir2021approximately} but there are several important differences. Our main tool for controlling \ref{eqn:main.expression} is Lemma \ref{lem:expectation}, which in \cite{yakir2021approximately} is proved only for the Rademacher case and relies on a version of Lemma \ref{lem:small.ball.probability} tailored to that setting. The Lemma \ref{lem:small.ball.probability} is a so--called small--ball probability statement, its proof uses the classic Essen Lemma, Tur\'an's Lemma as in \cite{yakir2021approximately}, here we only adding some probabilistic ideas to extend for more variety of random variables. In order to handle more general distribution, the proof of Lemma \ref{lem:expectation} here requiere an estimate for the the expected Mahler measure \cite{smyth2007mahler} that is no treated in \cite{yakir2021approximately} and is of independent interest.

The paper is organized as follows. In Section \ref{sec:main.results} we state the main result of this work, which formalizes several of the ideas discussed in this introduction. In Section \ref{sec:num.zeros.inside.unit.circle} we state and prove a result showing that, with high probability, approximately one half of the zeros of a random polynomial lie inside the unit disk when its degree is sufficiently large, for random coefficients satisfying the hypothesis collected in $\mathcal{H}$ in Section \ref{sec:main.results}. In Section \ref{sec:expectation.mahler.measure} we analyze the Mahler measure for random polynomials, and in Section \ref{sec:small.ball.probability} we address a small--ball probability problem which permits to prove the main results in Section \ref{sec:num.zeros.inside.unit.circle}.\\

{\bf Acknowledgments}. Before proceeding, I would like to thank Víctor Pérez-Abreu for his valuable comments, and Michelle Zela, Marceline X., Kusy X., Yanay X., and Sophie X. for their support during the preparation of this work.


\section{Main Results}\label{sec:main.results}

Before stating our main results, we introduce some notation. Let $\mu$ denote the normalized measure on $[-\pi,\pi]$. We write $f_n=\textnormal{O}(g_n)$ if there exist a constant $C>0$ such that $f_n\leq C g_n$ for all sufficiently large $n$, and $f_n=\textnormal{o}(g_n)$ if $f_n/g_n\to0$ as $n\to\infty$. We write $f_n \sim g_n$ if $f_n/g_n \to 1$ as $n\to\infty$. Throughout, $c, C, C_{*},C_1, C_2,\ldots$ denote positive constants whose values may change from line to line. We use the symbol $\propto$ to denote some proporcional relation. 

Given a banded Toeplitz matrix $T_n \in C^{n\times n}$, we also consider its infinite--dimensional counterpart $T$ (viewed as a Toeplitz operator) version $T$ (operator) of it, i.e.,
\[
T :=
\CC{\begin{array}{ccccc}
\xi_{0} & \ldots & \xi_{r} &  &  \\
\vdots & \ddots &  & \ddots& \\
\xi_{-r} &  & \ddots & & \ddots \\
 & \ddots &  & \ddots & \\
\end{array}}.
\]
We write $T_n(P)$ (respectively, $T(P)$) to emphasize its relation with the Laurent polynomial $P$, which is commonly referred to as the {\it symbol} of $T_n$ (respectively, $T$).

We consider a Kac (random) polynomial \[ G(z) =\sum_{j=0}^{n} \xi_j z^j,\quad n\geq 2,\] whose coefficients satisfy the following standing assumptions $\mathcal{H}$:
\begin{itemize}
\item[($\mathcal{H}1$)] $\xi_j$ are independent identically distributed (iid) non--degenerate random variables.
\item[($\mathcal{H}2$)] $\E{\xi_{0}}=0$, $\E{\xi_{0}^2} = 1$, $\E{\xi_{0}^4} < +\infty$.
\item[($\mathcal{H}3$)] Anti--concentration near zero: There exists $C_1>0$ such that \[ \Pr{\abs{\xi_0}\leq t} \leq C_1 t^{1/2},\quad t\in[0,1]. \]
\item[($\mathcal{H}4$)] Let $\xi_0^{'}$ be an independent copy of $\xi_0$, and define \[ \overline{\xi}_0:=(\xi_0^{'}-\xi_0)\mathds{1}_{\LL{\xi_0^{'} \neq \xi_0}}. \] Then $\overline{\xi}_0$ also satisfies an anti--concentration near zero: there exists $C_2>0$ such that \[ \Pr{\abs{\overline{\xi}_0}\leq t} \leq C_2 t^{1/2}, \quad t\in[0,1]. \]
\item[($\mathcal{H}5$)] There exist $M>0,\gamma>0$ and $q\in(0,1)$ such that \[\sup_{u\in\R}\Pr{\abs{\xi_0 - u}\leq \gamma} \leq 1- q, \quad \Pr{\abs{\xi_0}>M} \leq \frac{q}{2}.\]
\end{itemize}
\vspace{1em}
Note that $\mathcal{H}3$ implies $\Pr{\xi_0=0}=0$.  Observe that both Radamecher and Standard Gaussian random variables satisfy the previous assumptions. More generally, any non--degenerate random variable $\xi_0$  with a symmetric distribution such that $\Pr{\abs{\xi_0}\geq c} = 1$ for some $c>0$ satisfies $\mathcal{H}3$ and $\mathcal{H}4$. Indeed, \[\Pr{\abs{\xi_0}\leq t} \leq c^{-1/2} t^{1/2},\quad t\in[0,1].\] Moreover, since $\xi_0$ takes at least two distinct values, the difference $\xi_0'-\xi_0$ takes at least four values, where $\xi_0'$ is an independent copy $\xi_0$. In particular, it attains nonzero values with $\abs{\xi_0'-\xi_0} \geq c$, which yields the required anti--concentration bound in $\mathcal{H}4$.

The main result of this work is the following.

\begin{theorem}\label{thm:main}
Let $r,s$ be nonnegative integers and set  $m:=r+s$. Let \[ G_{m}(z)=\sum_{j=0}^{m} \xi_j z^j \quad P_{r,s}(z)=z^{-r}G_{m}(z). \] We assume that $P_{r,s}$ is a random Laurent polynomial whose random coefficients $\xi_j$ satisfy assumptions $\mathcal{H}$. For each realization of $G_m$, let $\LL{T_n : n\in\N}$ denote the associated sequence of $n\times n$ banded Toeplitz matrices with symbol $P_{r,s}$. We have:
\begin{itemize}
\item If $r=s$, 
\begin{align*}
& \lim_{n\to\infty} \sigma_n(T_n(P_{r,s})) = \norm{T^{-1}(P_{r,s})} \geq c_1 \CC{ m^2 (\log m)^3}^{-m}, \\
& \lim_{n\to\infty} \sigma_1(T_n(P_{r,s})) = \norm{T(P_{r,s})} \leq c_2 m^{1/2}\log m,
\end{align*}
with probability  
\[
1 - c_3 - \textnormal{O}\PP{(\log m)^{-1/2}},
\]
for some constants $c_1,c_2>0$, $c_3\in(0,1)$. All constant, including the implicit constant in the big--\textnormal{O} notation, depend only on the distribution $\xi_0$.

\item If $r\neq s$, then 
\[
\norm{T_n^{-1} (P_{r,s})}\geq c_4 \CC{2m^{3}(\log m)^{4}}^{-m} e^{ \alpha_m n }
\]
with probability 
\[
1 - c_5  - \textnormal{O}\PP{(\log m)^{-1/2}},
\] 
where $\alpha_m:=\log\PP{1 + \frac{1}{2m^2(\log m)^3}}$ and for some constants $c_4>0$, $c_5\in(0,1)$. All constants, incluiding the implicit constant in the big--\textnormal{O} notation, depend only on the distribution of $\xi_0$.
\end{itemize}
\end{theorem}

\begin{remark} 
As it will be observed in the proof of Theorem \ref{thm:main}, the constants $c_3,c_5$ can be small but they cannot be arbitrarily small. In fact, there exist the relation $c_1 \propto \frac{1}{c_3}$  and $c_4 \propto \frac{1}{c_5}$.
\end{remark}

\begin{proof}[Proof of Theorem \ref{thm:main}]
We observed that $\textnormal{wind}(P)=0$ is necessary that $r=\nu_m$, where $\nu_m$ is the number of zeros of $G_m(z)$ which are inside of the unit circle. From Theorem \ref{thm:main01}, we have with probability $1-\textnormal{O}\PP{ {(\log m)^2}{m^{-1/10}} }$, where the implicit constant depends on the distribution $\xi_0$, that $r$ might satisfy the relation
\[
\frac{m}{2} - \frac{1}{m^{1/10}} \leq r \leq \frac{m}{2} + \frac{1}{m^{1/10}},
\]
i.e., the number of possible integer values of $r$ is at most
\[
\PP{ \frac{m}{2} + \frac{1}{m^{1/10}} } - \PP{  \frac{m}{2} - \frac{1}{m^{1/10}} } + 1 = \frac{2}{m^{1/10}} + 1.
\]

If $m$ is sufficiently large, we have $r=\frac{m}{2}$, i.e., $r=s$. Thus, if initially $r=s$, we have (see expression 1.12 in \cite{bottcher2005spectral})
\[
\textnormal{wind}(P)= \frac{2r}{2} - r = 0,
\]
with probability 
\[
1-\textnormal{O}\PP{ {(\log m)^2}{m^{-1/10}} }.
\]
From Theorem 1.1 in \cite{barrera2022zero}, we have that the nearest $z^{*}$ of $G_m(z)$ to the unit circle satisfy
\[
\abs{z^{*}} \leq 1 - \frac{1}{m^2(\log m)^3} \quad \mbox{or} \quad 1 + \frac{1}{m^2(\log m)^3} \leq \abs{z^{*}}
\]
with probability
\[
1-\textnormal{O}\PP{(\log m)^{-1/2}},
\]
where the implicit constant depends on the distribution $\xi_0$. If $r=s$, we have from Corollary 6.5 (see expression 6.16) in \cite{bottcher2005spectral} that $T$ is invertible operator such that
 \begin{align*}
& \lim_{n\to\infty} \sigma_n(T_n) = \norm{T^{-1}(P)}, \\
& \lim_{n\to\infty} \sigma_1(T_n) = \norm{T(P)} = \max_{t\in\mathbf{T}} \abs{G_{m}(t)}, \\
& \lim_{n\to\infty} \kappa(T_n)  = \kappa(T),
\end{align*}
where $\kappa(T)$ is defined analogously to the finite-dimensional case.

From Section {\it Taylor’s approximation} in \cite{barrera2022zero}, 
\begin{equation}\label{eqn:max.pol}
\Pr{\max_{t\in\mathbf{T}} \abs{G_{m}(t)} \leq C_0 m^{1/2}\log m  } = 1 -  \textnormal{O}\PP{(\log m)^{-1/2}},
\end{equation}
for some adequate positive constant $C_0$ and the implicit constant in the big--O notation depends only on the distribution of $\xi_0$.

From the Section 1.7 in \cite{bottcher2005spectral} we have that $P(t) = P_{-}(t) P_{+}(t)$ with
\begin{equation}\label{eqn:facpoly}
P_{-}(t) = \prod_{j=1}^{r} \PP{ 1 - \frac{\delta_j}{t} },\quad P_{+}(t) = \xi_s \prod_{k=1}^{s} \PP{ t - \mu_k },
\end{equation}
where $\delta_j, \mu_k$ are the zeros of $G_m(z)$ and $\abs{\delta_j} < 1$ and $\abs{\mu_k}>1$. Let 
\begin{equation}\label{eqn:delta.mu}
\delta = \max\LL{\abs{\delta_j} : j=1,\ldots,r}, \quad \mu:=\min\LL{\abs{\mu}: k=1,\ldots, s}.
\end{equation}
From expression 1.20 in \cite{bottcher2005spectral}, we have that
\[
T^{-1}(P) = T(P_{-}^{-1}) T(P_{+}^{-1}),
\]
then
\[
\norm{ T^{-1}(P) } \leq \norm{ T(P_{-}^{-1}) } \norm{ T(P_{+}^{-1}) }.
\]
From Theorem 1.5 in \cite{bottcher2005spectral},
\begin{align*}
& \norm{ T(P_{+}^{-1}) } = \max_{t\in\mathbf{T}} \abs{P_{+}^{-1}} = \PP{\min_{t\in\mathbf{T}} \abs{P_{+}} }^{-1}, \\
& \norm{ T(P_{-}^{-1}) } = \max_{t\in\mathbf{T}} \abs{P_{-}^{-1}} = \PP{\min_{t\in\mathbf{T}} \abs{P_{-}} }^{-1}.
\end{align*}
From expression \ref{eqn:facpoly} and the previous observations, 
\begin{align*}
& \norm{ T(P_{+}^{-1}) } \leq \frac{1}{(1-\delta)^r} \leq \CC{ m^2(\log m)^3 }^{r}, \\
& \norm{ T(P_{-}^{-1}) } \leq \frac{1}{\abs{\xi_s}(\mu - 1)^s} \leq \frac{1}{c} \CC{ m^2 (\log m )^3 }^{s},
\end{align*}
where the last line happens with probability (see hypothesis $\mathcal{H}3$)
\[
\Pr{\abs{\xi_s} > c } = 1 - C_1 c^{1/2} > 0,
\]
and $c\in(0,1)$. Then,
\[
\norm{T^{-1}(b)} \leq \frac{1}{c} \CC{m^2 (\log m)^{3} }^{m}
\]
with probability
\[
1 - C_1 c_0^{1/2} - \textnormal{O}\PP{ (\log m)^2 m^{-1/10} } - \textnormal{O}\PP{(\log m)^{-1/2}}.
\]

If $r\neq s$, then 
\[
\textnormal{wind}(P) \neq 0
\]
with probability
\[
1-\textnormal{O}\PP{ {(\log m)^2}{m^{-1/10}} }.
\]
We consider the Wiener--Hopf factorization of $P(z)$ (see Theorem 1.8 in \cite{bottcher2005spectral}),
\[
P(t) = P_{-}(t) t^{\textnormal{wind}(P)} P_{+}(t), \quad t\in\mathbf{T},
\]
where $P_{-}(t), P_{+}(t)$ are as defined in \ref{eqn:facpoly}. From Lemma 1.7 in \cite{bottcher2005spectral}, we have 
\[
P_{-}^{-1}(t) = \sum_{j=0}^{\infty} \frac{q_n^{-}}{t^j}, \quad P_{+}^{-1}(t) = \sum_{k=0}^{\infty} q^{+}_j t^j,
\]
such that
\begin{align*}
\abs{q^{-}_{n}} & \leq \PP{ \min_{\abs{z}=\delta+\varepsilon} \abs{P_{-}(z)} }^{-1} \PP{\delta+\varepsilon}^n \quad \mbox{ for $\varepsilon > 0$}, \\
\abs{q^{+}_{n}} & \leq \PP{ \min_{\abs{z}=\mu-\varepsilon} \abs{P_{+}(z)} }^{-1} \PP{\mu - \varepsilon}^{-n} \quad \mbox{ for $0< \varepsilon < \mu$}.
\end{align*}
where $\delta, \mu$ are as defined in \ref{eqn:delta.mu}. For the above observations, we observa that
\[ 
\min_{\abs{z}=\delta+\varepsilon} \abs{P_{-}(z)} \geq \PP{ \frac{\varepsilon}{\delta + \varepsilon} }^{r}, \quad \min_{\abs{z}=\mu-\varepsilon} \abs{P_{+}(z)} \geq \abs{\xi_s} {\varepsilon}^{s},
\]
and consequently
\begin{align}
\abs{q^{-}_{n}} & \leq \PP{ \frac{\delta + \varepsilon}{\varepsilon} }^{r} \PP{\delta + \varepsilon}^{n} \leq \varepsilon^{-r} \PP{\delta + \varepsilon}^{n+r} \leq \varepsilon^{-r} \PP{1 - \frac{1}{m^2 (\log m)^3} + \varepsilon}^{n+r} , \label{eqn:qminus}\\
\abs{q^{+}_{n}} & \leq \frac{1}{\abs{\xi_s}} \varepsilon^{-s} \PP{\mu - \varepsilon}^{-n} \leq \frac{1}{c}\varepsilon^{-s} \PP{\mu - \varepsilon}^{-n} \leq \frac{1}{c} \varepsilon^{-s} \PP{1 + \frac{1}{m^2 (\log m)^{3}} - \varepsilon}^{-n}, \label{eqn:qplus}
\end{align}
where the last line happens with probability (see hypothesis $\mathcal{H}3$)
\[
\Pr{\abs{\xi_s} > c } = 1 - C_1 c^{1/2} > 0,
\]
and $c\in(0,1)$.
Let $0 < \varepsilon < \mu\leq 1 + \frac{1}{m^2(\log m)^3}$ be fixed sufficiently small such that the quantities \ref{eqn:qminus} and \ref{eqn:qplus} decreasing exponentially fast to zero. For example, $\varepsilon_m := \frac{1}{2m^2(\log m)^3}$. From the proof of Theorem 4.1 in \cite{bottcher2005spectral}, if $\textnormal{wind}(P) < 0 $ we can consider $x^{(n)}=(q^{+}_{0},\ldots,q^{+}_{n-1})^{T}$ (or $x^{(n)}=(q^{-}_{0},\ldots,q^{-}_{n-1})^{T}$ if $\textnormal{wind}(P)>0$) and we observe from \ref{eqn:max.pol} that
\[
\norm{ T_n x^{(n)} }_2 \leq C_2 \CC{2m^2(\log m)^3}^{m} \norm{T_n}_2 e^{-\alpha_m n} \leq C_2  \CC{2m^{3}(\log m)^{4}}^{m} e^{-\alpha_m n},
\]
with probability 
\[
1 -  C_1c^{1/2} - \textnormal{O}\PP{(\log m)^{2} m^{-1/10}} - \textnormal{O}\PP{(\log m)^{-1/2}},
\]
for some positive constant $C_2$ which depends only on the distribution of $\xi_0$ and
\[
\alpha_m := \log\PP{ 1 + \frac{1}{2m^2(\log m)^3}},
\]
since 
\[
\abs{q^{+}_{n}}  \leq \frac{1}{c}\CC{2m^2(\log m)^3}^{s} e^{-\alpha_m n} ,\quad
\abs{q^{-}_{n}}  \leq \CC{2m^2(\log m)^3}^{r} e^{-\alpha_m n},
\]
where the second relation is from the fact $\log(1- \varepsilon_m)\leq -\log(1+ \varepsilon_m)$ and hence
\[
(n+r)\log(1 - \varepsilon_m) \leq n\log(1 - \varepsilon_m) \leq -n\log(1+ \varepsilon_m) = -\alpha_m n
\]
\end{proof}

Theorem \ref{thm:main} shows that when $r\neq s$ the sequence $\LL{T_n : n\in\N}$ of banded Toeplitz matrix is not stable, in particular, for all sufficiently large $n$, the matrices $T_n$ are typically ill conditioning. This behavior is somewhat surprising in light of the classic theory for random matrices with iid entries, where the probability of singularity tends to $0$ as the dimension grows, and quantitative bounds on extreme singular vales yield good condtioning with high probability \cite{barrera2022asymptotic,livshyts2021smallest,pan2015estimating,rudelson2010non}. In the case of a circulant matrix, whose is a special subclass of Toeplitz matrix, the condition number is polynomial when its entries are iid random variable \cite{barrera2022asymptotic}. 

\section{Number of zeros of a random polynomial inside the unit circle} \label{sec:num.zeros.inside.unit.circle}

Here we establish the second main result in this work. We prove that, for a random polynomial whose coefficients satisfy assumptions $\mathcal{H}1$--$\mathcal{H}4$, asymptotically one half of its zeros lie inside the unit circle. This results is of independent interest.

\begin{theorem}\label{thm:main01}
Let $G(z)=\sum_{j=0}^{n} \xi_j z^j$ where its coefficient satisfy $\mathcal{H}1$--$\mathcal{H}4$, then
\[
\Pr{\abs{\frac{\nu_n(\mathds{D})}{n} - \frac{1}{2} } \geq n^{-1/10} } =  \textnormal{O}\PP{ \frac{(\log n)^2}{n^{1/10}} },
\]
where the implicit constant depends on the distribution of $\xi_0$. In particular, $\frac{\nu_n(\mathds{D})}{n} \to \frac{1}{2}$ in probability. 
\end{theorem}


In order to prove Theorem \ref{thm:main01}, we show 
\begin{equation}\label{eqn:upperbound}
\Pr{\nu_n(\mathds{D}) \geq \frac{n}{2}+n^{9/10}} = \textnormal{O}\PP{ \frac{(\log n)^2}{n^{1/10}} },
\end{equation}
and 
\begin{equation}\label{eqn:lowerbound}
\Pr{\nu_n(\mathds{D}) \leq \frac{n}{2}-n^{9/10}}= \textnormal{O}\PP{ \frac{(\log n)^2}{n^{1/10}}}.
\end{equation}
where the implicit positive constant depends on the distribution of $\xi_0$.

From Jensen's formula (p. 207 in \cite{ahlfors1979complex}) states that
\begin{equation}\label{eqn:jensen}
\int_{0}^r \frac{N_f(t)}{t} \dd t = \int_{-\pi}^{\pi} \log\abs{f(re^{i\theta})} \dd\mu(\theta) - \log\abs{f(0)}
\end{equation}
where $f$ is analytic function on $\LL{\abs{z}\leq r}$ such that $f(0)\neq 0$ and 
\[
N_f(t) := \#\LL{\abs{z}\leq t : f(z)=0}.
\]

\begin{proof}[Proof of Theorem \ref{thm:main01}]
We establish the upper bound \ref{eqn:upperbound}. Let $\tau=n^{-11/10} > 0$. Observe
\begin{align}
\nu_n(\mathds{D}) & \leq \frac{1}{\log(1+\tau)} \int_{1}^{1+\tau} \frac{\nu_n(r\mathds{D})}{r} \dd r \nonumber\\
& = \frac{1}{\log(1+\tau)} \CC{\int_{-\pi}^{\pi} \log \abs{G_n((1+\tau)e^{i\theta})} \dd \mu(\theta) - \int_{-\pi}^{\pi} \log \abs{G_n(e^{i\theta})} \dd \mu(\theta)} \nonumber\\
& = \frac{\log\sigma(1+\tau)-\log\sigma(1)}{\log(1+\tau)} \nonumber\\ 
& \quad\quad + \frac{1}{\log(1+\tau)} \CC{\int_{-\pi}^{\pi} \log \abs{\tilde{G}_n((1+\tau)e^{i\theta})} \dd \mu(\theta) - \int_{-\pi}^{\pi} \log \abs{\tilde{G}_n(e^{i\theta})} \dd \mu(\theta)}.
\end{align}
We define 
\[
S(r):=\sum_{j=0}^n r^{2j},\quad h(r):= \log S(r),
\]
thus by definition $\sigma(r)$, we have $\log \sigma(r) = \frac{1}{2}h(r)$. Now, we see the derivatives of $S$ with respect $r$ up to order three are
\[
S'(r) = \sum_{j=0}^n 2jr^{2j-1},  \quad S''(r)  = \sum_{j=0}^n 2j(2j-1)r^{2j-2}, \quad S''(r)  = \sum_{j=0}^n 2j(2j-1)(2j-2)r^{2j-3},
\]
and for $h(r)$,
\begin{align*}
h'(r) & = \frac{S'}{S}, \\[6pt]
h''(r) & = \frac{S''S-(S')^2}{S^2}, \\[6pt]
h'''(r) & = \frac{S^2S''' - 3 SS'S'' + 2(S')^3}{S^3}.
\end{align*}
If $r\in[1-\tau,1+\tau]$, then all sufficiently large $n$,
\begin{itemize}
\setlength{\itemsep}{6pt}
\item $e^{-4} n \leq S(r) \leq e^2 n$, 
\item $\abs{S'(r)} \leq 2e^2 \sum_{j=0}^n j = \textnormal{O}(n^2)$,
\item $\abs{S''(r)} \leq 4e^2 \sum_{j=0}^n j^2 = \textnormal{O}(n^3)$,
\item $\abs{S'''(r)} \leq 8e^2 \sum_{j=0}^n j^3 = \textnormal{O}(n^4)$.
\end{itemize}

From the above observation, we have
\begin{align*}
\left. \frac{\partial}{\partial r} \log\sigma(r) \right|_{r=1} & = \frac{1}{2} n, \\[6pt]
\left. \frac{\partial^2}{\partial r^2} \log\sigma(r) \right|_{r=1} & = \frac{1}{2} \frac{n(n-1)}{3}, \\[6pt]
\abs{ \frac{\partial^2}{\partial r^2} \log\sigma(r) } & =  \frac{1}{2}\abs{h'''(r)} = \textnormal{O}(n^3), 
\end{align*}
uniformly for all $r\in[1-\tau,1+\tau]$.
Using a third order of Taylor expansion for $\log\sigma(r)$ ,
\begin{align*}
2\log\sigma(1+r) & = 2\log\sigma(1) + r h'(1) + \frac{1}{2}r^2 h''(1) + \textnormal{O}(h'''(r^*)r^3) \\
& = 2\log\sigma(1)+ nr +\frac{n(n-1)}{6} r^2 + \textnormal{O}(n^3r^3),
\end{align*}
for some $r^*\in[1,1+r]$ if $r>0$ or $r^*\in[1+r,1]$ if $r<0$.

Using the fact that $\log(1+\tau) \sim \tau$ for $\tau$ small, yields,
\[
\frac{\log\sigma_n(1+\tau)-\log\sigma(1)}{\log(1+\tau)} -\frac{n}{2} = \frac{1}{12}\tau n^2 + \textnormal{O}(n^3\tau^2) = \frac{1}{12}n^{9/10} + \textnormal{O}(n^{8/10}),
\]
and
\[
\frac{\log\sigma_n(1)-\log\sigma(1+\tau)}{-\log(1-\tau)} -\frac{n}{2} = -\frac{1}{12}n^{9/10} + \textnormal{O}(n^{8/10}).
\]

Thus, using the Markov inequality and the previous observations we note that
\begin{align*}
& \Pr{\nu_n(\mathds{D}) \geq \frac{n}{2}+n^{9/10}} \\
& \quad \leq \Pr{{\int_{-\pi}^{\pi} \log \abs{\tilde{G}_n((1+\tau_n)e^{i\theta})} \dd \mu(\theta) - \int_{-\pi}^{\pi} \log \abs{\tilde{G}_n(e^{i\theta}) } \dd \mu(\theta)} \geq \frac{11}{12}\tau n^{9/10} + \textnormal{O}(\tau n^{8/10}) } \\
& \quad = \Pr{{\int_{-\pi}^{\pi} \log \abs{\tilde{G}_n((1+\tau_n)e^{i\theta})} \dd \mu(\theta) - \int_{-\pi}^{\pi} \log \abs{\tilde{G}_n(e^{i\theta}) } \dd \mu(\theta)} \geq \frac{11}{12} n^{-1/5} + \textnormal{O}(n^{-3/10})} \\
& \quad \leq \Pr{\abs{\int_{-\pi}^{\pi} \log \abs{\tilde{G}_n((1+\tau_n)e^{i\theta})} \dd \mu(\theta) + \frac{\gamma}{2}} \geq \frac{11}{12} n^{-1/5} + \textnormal{O}(n^{-3/10}) } \\
& \quad\quad + \Pr{\abs{\int_{-\pi}^{\pi} \log \abs{\tilde{G}_n((1)e^{i\theta})} \dd \mu(\theta) + \frac{\gamma}{2}} \geq \frac{11}{12} n^{-1/5} + \textnormal{O}(n^{-3/10}) } \\
& \quad \leq  C_1 \PP{\frac{11}{12} n^{-1/5} + \textnormal{O}(n^{-3/10})}^{-2} \times \frac{(\log n)^2}{n^{1/2}} \leq C_2 \frac{(\log n)^2}{n^{1/10}},
\end{align*}
for some positive constante $C_2$ which depends on the distribution of $\xi_0$.

The lower bound \ref{eqn:lowerbound} is obtained in a similar way, considering now the inequality
\begin{align}
\nu_n(\mathds{D}) & \geq \frac{1}{-\log(1-\tau_n)} \int_{1-\tau_n}^{1} \frac{\nu_n(r\mathds{D})}{r} \dd r \nonumber\\
& = \frac{1}{-\log(1-\tau_n)} \CC{\int_{-\pi}^{\pi} \log \abs{G_n(e^{i\theta})} \dd \mu(\theta) - \int_{-\pi}^{\pi} \log \abs{G_n((1-\tau_n)e^{i\theta})} \dd \mu(\theta)} \nonumber\\
& = \frac{\log\sigma_n(1)-\log\sigma_n(1-\tau_n)}{-\log(1-\tau_n)} \nonumber\\ 
& \quad\quad +\frac{1}{-\log(1-\tau_n)} \CC{\int_{-\pi}^{\pi} \log \abs{\tilde{G}_n(e^{i\theta})} \dd \mu(\theta) - \int_{-\pi}^{\pi} \log \abs{\tilde{G}_n((1-\tau_n)e^{i\theta})} \dd \mu(\theta)}.
\end{align}
Thus,
\begin{align*}
& \Pr{\nu_n(\mathds{D}) \leq \frac{n}{2} - n^{9/10}} \\
& \quad \leq \Pr{\int_{-\pi}^{\pi} \log \abs{\tilde{G}_n(e^{i\theta})} \dd \mu(\theta) - \int_{-\pi}^{\pi} \log \abs{\tilde{G}_n((1-\tau_n)e^{i\theta})} \dd \mu(\theta) \leq -\frac{11}{12}\tau n^{9/10}+\textnormal{O}(\tau n^{8/10}) } \\
& \quad \leq \Pr{\int_{-\pi}^{\pi} \log \abs{\tilde{G}_n(e^{i\theta})} \dd \mu(\theta) - \int_{-\pi}^{\pi} \log \abs{\tilde{G}_n((1-\tau_n)e^{i\theta})} \dd \mu(\theta) \leq  -\frac{11}{12}n^{-1/5}+\textnormal{O}(n^{-3/10}) } \\
& \quad \leq \Pr{\abs{\int_{-\pi}^{\pi} \log \abs{\tilde{G}_n((1-\tau_n)e^{i\theta})} \dd \mu(\theta) - \int_{-\pi}^{\pi} \log \abs{\tilde{G}_n(e^{i\theta})} \dd \mu(\theta) } \geq\frac{11}{12}n^{-1/5}+\textnormal{O}(n^{-3/10}) } \\
& \quad \leq \Pr{\abs{\int_{-\pi}^{\pi} \log \abs{\tilde{G}_n((1-\tau_n)e^{i\theta})} \dd \mu(\theta) + \gamma } \geq \frac{11}{12}n^{-1/5}+\textnormal{O}(n^{-3/10}) } \\
& \quad\quad +  \Pr{\abs{\int_{-\pi}^{\pi} \log \abs{\tilde{G}_n((1)e^{i\theta})} \dd \mu(\theta) + \gamma } \geq  \frac{11}{12}n^{-1/5}+\textnormal{O}(n^{-3/10}) } \\
& \quad \leq C_1 \PP{ \frac{11}{12}n^{-1/5}+\textnormal{O}(n^{-3/10}) }^{-2} \times \frac{(\log n)^2}{n^{1/2}} \leq C_4 \frac{(\log n)^2}{n^{1/10}}
\end{align*}
for some positive constante $C_4$ which depends on the distribution of $\xi_0$.

\end{proof}

\begin{lemma}\label{lem:expectation} Let $m\in\{1,2\}$. For all $r\in [1-n^{-11/10},1+n^{-11/10}]$ we have that 
\[
\mathbb{E} \left[ \left( \int_{ -\pi}^{\pi} \log |\tilde{G}(re^{i\theta})| \textnormal{d}\mu(\theta) \right)^m \right] = \left( -\frac{\gamma}{2} \right)^m +  \textnormal{O}\PP{\frac{(\log n)^2 }{\sqrt{n}}},
\]
where $\gamma \approx 0.5772...$ is the Euler's constant and the implicit constant depends on the distribution of $\xi_0$.
\end{lemma}

Before of the proof of Lemma \ref{lem:expectation}, we introduce some notation and an auxiliary results that will be useful.

For every fixed $\theta\in[-\pi,\pi]$ denote by
\[
F^{\theta}_n(x) := \Pr{ \abs{\tilde{G}(re^{i\theta})}^2 \leq x }, \quad F(x) := 1 - e^{-x}.
\]
For $\theta\neq \phi$ with $\theta,\phi\in [-\pi,\pi]$, we denote by
\[
F_{n}^{\theta,\phi}(x,y) := \Pr{ \abs{\tilde{G}(re^{i\theta})}^2 \leq x, \abs{\tilde{G}(re^{i\theta})}^2 \leq y }.
\]

\begin{lemma}\label{lem:dist_approx} There exist an universal constant $C>0$ such that

\begin{enumerate}
\item for all $\abs{\theta} \geq n^{-1/2}$ and for all $x\in(0,\infty)$ \[ \abs{F_{n}^{\theta}(x) - F(x)} \leq \frac{C}{\sqrt{n}}.\]
\item for all $\theta,\phi \not\in [-n^{-1/2},n^{-1/2}]$ such that $\abs{\theta - \phi} > n^{-1/2}$ \[ \abs{F_{n}^{\theta,\phi}(x,y) - F(x)F(y)} \leq \frac{C}{\sqrt{n}}\] for all $x,y\in(0,\infty)$.
\end{enumerate}
\end{lemma} 

Now, we proceed the proof of Lemma \ref{lem:expectation}.

\begin{proof}[Proof of Lemma \ref{lem:expectation}]
We only prove the case $m=2$, the left case is similar.  Denote $p(\theta) := \tilde{G}(re^{i\theta})$ and is defined the event \[ \mathcal{A}_\theta := \LL{ \abs{G(re^{i\theta})} \leq n^{-A} } \] for some large value $A$ which will be chose later. From Lemma \ref{lem:mahler}, we may apply Fubini Theorem and observe that
\begin{align*}
& \E{\PP{ \int_{-\pi}^{\pi} \log\abs{\tilde{G}(re^{i\theta})} \dd\mu(\theta) }^2} \\
& \quad = \E{\int\int_{[-\pi,\pi]^2} \log\abs{p(\theta)} \log\abs{p(\phi)} \dd\mu(\theta)\dd\mu(\phi) } \\
& \quad = \E{ \int\int_{[-\pi,\pi]^2}  L(\theta,\phi) \mathds{1}_{\mathcal{A}^c_\theta} \mathds{1}_{\mathcal{A}^c_{\phi}} \dd\mu(\theta) \dd\mu(\phi)} + 2\E{\int\int_{[-\pi,\pi]^2}  L(\theta,\phi) \mathds{1}_{\mathcal{A}_\theta} \mathds{1}_{\mathcal{A}^c_{\phi}} \dd\mu(\theta) \dd\mu(\phi)} \\
& \quad\quad + \E{\int\int_{[-\pi,\pi]^2}  L(\theta,\phi) \mathds{1}_{\mathcal{A}_\theta} \mathds{1}_{\mathcal{A}_{\phi}} \dd\mu(\theta)\dd\mu(\phi)},
\end{align*}
where $L(\theta,\phi):= \log\abs{p(\theta)} \log\abs{p(\phi)} $. We denote the summands of the previous sum by $E_1,E_2,E_3$, respectively. We show that $E_2, E_3$ are small.

By applying Tonelli's Theorem, Cauchy--Schwarz inequality, Lemma \ref{lem:mahler} and Lemma \ref{lem:small.ball.probability} with $c=5$ and $a=n^{-112}$ ($A=112$), we have 
\begin{align*}
\abs{E_3} & \leq \int\int_{[-\pi,\pi]^2}\E{ \abs{\log\abs{p(\theta)} \log\abs{p(\phi)}} \mathds{1}_{\mathcal{A}_\theta} \mathds{1}_{\mathcal{A}_{\phi}} } \dd\mu(\theta) \dd\mu(\phi) \\
& \leq \PP{\int_{-\pi}^{\pi} \sqrt{ \E{(\log\abs{p(\theta)})^2 \mathds{1}_{\mathcal{A}_\theta}} } \dd\mu(\theta)}^2 \\
& \leq \int_{-\pi}^{\pi} \sqrt{\E{(\log\abs{p(\theta)})^4} \Pr{\mathcal{A}_\theta}} \dd \mu(\theta) \\
& \leq \PP{ \int_{-\pi}^{\pi} \E{\abs{\log\abs{p(\theta)}}^4 } \dd\mu(\theta)}^{1/2} \times \PP{ \int_{-\pi}^{\pi} \Pr{\mathcal{A}_\theta} \dd\mu(\theta)}^{1/2} \\
& \leq \PP{(C^{*}+1) \PP{n\log n}^4}^{1/2} \times \PP{C_5 \frac{112\log n}{n^5}}^{1/2} \\
& \leq C_6 \frac{(\log n)^{5/2}}{\sqrt{n}}, 
\end{align*}
for some positive constant $C_6$ which depends on the distribution of $\xi_0$.

In the case of $E_2$, first we observe for all $n>1$ and $t>0$,
\[
\Pr{\abs{p(\phi)} > t} = \Pr{\abs{G(re^{i\phi})} > \sigma(r) t } \leq \frac{1}{t^2},
\]
thus
\begin{align*}
& \E{\abs{\log\abs{p(\phi)}}^2 \mathds{1}_{ \mathcal{A}_{\phi}^c} } \leq (112 \log n)^2 + \int_{0}^\infty \Pr{ \log\abs{p(\phi)} \geq \sqrt{t}, \abs{p(\phi)} > n } \dd t \\
& \quad \leq (112 \log n)^2 + \frac{(\log n)^2}{n^2} + \int_{(\log n)^2}^{\infty} e^{-2\sqrt{t}} \dd t \\
& \quad \leq (112 \log n)^2 + \frac{3(\log n)^2 + 1}{n^2}. 
\end{align*}
From the above, we have
\begin{align*}
& \abs{E_2}  \leq \int\int_{[-\pi,\pi]^2}\E{ \abs{\log\abs{p(\theta)} \log\abs{p(\phi)}} \mathds{1}_{\mathcal{A}_\theta} \mathds{1}_{\mathcal{A}_{\phi}^c} } \dd\mu(\theta) \dd\mu(\phi) \\
& \quad \leq \int\int_{[-\pi,\pi]^2} \PP{\E{\abs{\log\abs{p(\theta)}}^2 \mathds{1}_{\mathcal{A}_{\theta}}}}^{1/2} \PP{\E{\abs{\log\abs{p(\phi)}}^2 \mathds{1}_{\mathcal{A}_{\phi}^c}}}^{1/2} \dd\mu(\theta) \dd\mu(\phi) \\
& \quad = \PP{\int_{-\pi}^{\pi} \PP{\E{\abs{\log\abs{p(\theta)}}^2 \mathds{1}_{\mathcal{A}_{\theta}}}}^{1/2} \dd\mu(\theta) } \times \PP{ \int_{-\pi}^{\pi}  \PP{\E{\abs{\log\abs{p(\phi)}}^2 \mathds{1}_{\mathcal{A}_{\phi}^c}}}^{1/2} \dd\mu(\phi)} \\
& \quad \leq \PP{C_6 \frac{(\log n)^{5/2}}{\sqrt{n}}}^{1/2} \times \PP{(112 \log n)^2 +\frac{3(\log n)^2 + 1}{n^2}}^{1/2} \\
& \quad \leq C_7 \frac{(\log n)^{9/4}}{n^{1/4}},
\end{align*}
for some positive constant $C_7$ which depends on the distribution of $\xi_0$.

Now, we define the following event 
\[
D_n := \LL{ \abs{\theta - \phi}\leq n^{-1/2}} \cup \LL{\abs{\theta}\leq n^{-1/2}} \cup \LL{\abs{\phi}\leq n^{-1/2}} \subset [-\pi,\pi]^2.
\]
From the definition of $D_n$, 
\begin{align*}
& \E{ \int\int_{D_n} \abs{\log\abs{p(\theta)} \log\abs{p(\phi)}} \mathds{1}_{\mathcal{A}_{\theta}^c} \mathds{1}_{\mathcal{A}_\phi^c} \dd\mu(\theta) \dd\mu(\phi) } \\
& \quad \leq \int\int_{[-\pi,\pi]^2}\E{ \abs{\log\abs{p(\theta)} \log\abs{p(\phi)}} \mathds{1}_{\mathcal{A}_{\theta}^c} \mathds{1}_{\mathcal{A}_{\phi}^c} }  \mathds{1}_{D_n} \dd\mu(\theta) \dd\mu(\phi) \\
& \quad \leq \PP{(112 \log n)^2 + \frac{3(\log n)^2 + 1}{n^2}} \times \PP{ \int\int_{[-\pi,\pi]} \mathds{1}_{D_n} \dd\mu(\theta) \dd\mu(\phi)} \\
& \quad \leq \PP{(112 \log n)^2 +\frac{3(\log n)^2 + 1}{n^2}} \times \PP{\frac{3}{n^{1/2}}}\\
& \quad \leq C_8 \frac{(\log n)^2}{n^{1/2}},
\end{align*}
for some positive constant $C_8$ which depends on the distribution of $\xi_0$.

It is routine to check that 
\begin{align*}
\int_{n^{-A}}^n (\log(x)) dF(x) & = \int_{0}^{\infty} (\log(x))e^{-x} dx +  \mbox{O}_A\PP{\frac{\log n}{n^{A}}} \\
& = -\gamma +  \mbox{O}_A\PP{\frac{\log n}{n^{A}}}.
\end{align*}

If $(\theta,\phi)\not\in D_n$, we can apply item 2 from Lemma \ref{lem:dist_approx} (for the case $l=1$, we apply item 1) to see that
\begin{align*}
& \abs{ \E{\log \PP{\abs{p(\theta)}^2} \log\PP{\abs{p(\phi)}^2} \mathds{1}_{\mathcal{A}_{\theta}^c} \mathds{1}_{\mathcal{A}_{\phi}^c}} - \gamma^2} \\
& \quad \leq \int_{n^{-A}}^n \int_{n^{-A}}^n \abs{\log(x)\log(y)} \dd\PP{F_n^{\theta,\phi}(x,y)-F(x)F(y)} + \mbox{O}\PP{\frac{(\log n)^2}{n}} + \mbox{O}_A\PP{\frac{\log n}{n^{A}}} \\
& \quad \leq (A\log n)^2 \abs{\int_{n^{-A}}^n \int_{n^{-A}}^n \dd\PP{F_n^{\theta,\phi}(x,y)-F(x)F(y)}} + \mbox{O}\PP{\frac{(\log n)^2}{n}} + \mbox{O}_A\PP{\frac{\log n}{n^{A}}} \\
& \quad \leq \frac{C (A\log n)^2}{\sqrt{n}} + \mbox{O}\PP{\frac{(\log n)^2}{n}} + \mbox{O}_A\PP{\frac{\log n}{n^{A}}} \\
&\quad =\mbox{O}\PP{\frac{(\log n)^2 }{\sqrt{n}}}.
\end{align*}
where the implicit constant depends on the distribution of $\xi_0$. This completes the proof.
\end{proof}

\begin{proof}[Proof of Lemma \ref{lem:dist_approx}]
Recall that $\tilde{G}(z) = \frac{1}{\sigma(r)} G(z)$ where $G(z) = \sum_{j=0}^{n} \xi_j z^k$, $\sigma(r) = \E{\abs{G(re^{i\theta})}^2}$, $r\in [1-n^{-11/10}, 1+ n^{-11/10}]$. We assume that $\LL{\xi_j}$ is a set iid random variables with $\E{\xi_j}=0$ and $\E{\xi_j}=1$ (assumption $\mathcal{H}2$). We define the following random vectors
\[
p(\theta) = \frac{1}{\sigma(r)} \PP{ \sum_{j=0}^{n} \xi_j r^j \cos(k\theta), \sum_{j=0}^{n} \xi_j r^{j} \sin(j\theta) } \in \R^2
\]
and 
\[
(p(\theta),p(\phi)) =  \frac{1}{\sigma(r)} \PP{\sum_{j=0}^{n} \xi_j r^j \cos(k\theta),\sum_{j=0}^{n} \xi_j r^j \sin(k\theta), \sum_{j=0}^{n} \xi_j r^j \cos(k\phi), \sum_{j=0}^{n} \xi_j r^j \sin(k\phi) } \in \R^4.
\]
Note the covariance structure of them are given respectively by
\begin{equation}\label{eqn:var_r2}
V(\theta) = \frac{1}{\sigma(r)^2} \sum_{j=0}^{n} r^{2j} 
\CC{
\begin{array}{cc}
\cos^2(j\theta) & \cos(j\theta)\sin(j\theta)\\
 \cos(j\theta)\sin(j\theta) & \sin^2(j\theta)
\end{array}
},
\end{equation}
and
\begin{align}\label{eqn:var_r4}
& V(\theta,\phi) = \frac{1}{\sigma(r)^2} \sum_{j=0}^{n} r^{2j}  \\
& \times \CC{
\begin{array}{cccc}
\cos^2(j\theta) & \cos(j\theta)\sin(j\theta) & \cos(j\theta)\cos(j\phi) & \cos(j\theta)\sin(j\phi) \\
\cos(j\theta)\sin(j\theta) & \sin^2(j\theta) & \sin(j\theta)\cos(j\phi) & \sin(j\theta)\sin(j\phi) \\
\cos(j\phi)\cos(j\theta) & \cos(j\phi)\sin(j\theta) & \cos^2(j\phi) & \sin(j\phi)\cos(j\phi) \\
\sin(j\phi)\cos(j\theta) & \sin(j\phi)\sin(j\theta) & \sin(k\phi)\cos(j\phi) & \sin^2(j\phi)
\end{array}
}. \nonumber
\end{align}

From $V(\theta)$ and $V(\theta,\phi)$, we can observe that the proof of this lemma  follows the same strategy in the proof Lemma 3.2 in \cite{yakir2021approximately}.
\end{proof}

\section{Bounding expectation of Mahler measure} \label{sec:expectation.mahler.measure}

Recall that for a monic polynomial $F(z)=\prod_{j=1}^n (z-\alpha_j)$ the {\it Mahler measure} of $F$ is given by
\[
M(F):=\prod_{j=1}^{n} \max\LL{1,\abs{\alpha_j}}.
\]
Using Jensen's formula, we see that
\[
\log M(F) = \int_{-\pi}^{\pi} \log \abs{F(e^{i\theta})} \dd\mu(\theta).
\]

Using Jensen's inequality, the following result establishes an upper bound on the fourth moment of logarithm of the Mahler measure of $G$. For more details on the Mahler measure see \cite{smyth2007mahler}.

\begin{lemma}\label{lem:mahler}
Let $G(re^{i\theta})=\sum_{j=0}^{n} \xi_j r^{j} e^{ij\theta}$ for $\theta\in[-\pi,\pi]$ where $r\in[1-n^{-11/10},1+n^{-11/10}]$ is fixed and $\xi_j$, $j=0,\ldots,n$ are iid random variables with $\E{\xi_0}=0$, $\E{\xi_0^2}=1$, $\E{\xi_0^4}<+\infty$ and $\Pr{\abs{\xi_0}\leq t}\leq C t^{1/2}$ for some constant $C>0$ and $t\in[0,1]$. Then, there is a positive constant $C^{*}$ which depends only on the distribution of $\xi_0$ such that
\[
\E{ \int_{-\pi}^{\pi} \abs{\log\abs{G(re^{i\theta})}}^4 \dd\mu(\theta)} \leq C^{*} \PP{n\log n}^4.
\]
\end{lemma}

\begin{proof}
If $\norm{ G }_{\infty} := \sup_{\theta\in[-\pi,\pi]} \abs{G(re^{i\theta})}$ and $\norm{ G }_{2} :=\int_{-\pi}^{\pi} \abs{G(re^{i\theta})}^2 d\mu(\theta)$. We observe $\norm{G}_2^2 = \sum_{j=0}^{n} \abs{\xi_j}^2 r^{2j}$, $\norm{ G }_{2} \leq \norm{ G }_{\infty}$ and by Cauchy-Schwarz
\[
\abs{G(re^{i\theta})} \leq \sum_{j=0}^n \abs{\xi_j} r^{j} \leq \sqrt{n+1}\PP{\sum_{j=0}^n  \abs{\xi_j}^2 r^{2j} }^{1/2} = \sqrt{n+1} \norm{G}_2,
\]
i.e., $\norm{ G }_{\infty} \leq \sqrt{n+1} \norm{G}_2$.

Before to proceding with the proof of Lemma \ref{lem:mahler}, we state the classical Tur\`an--Nazarov Lemma (see Theorem 1.5 in \cite{nazarov1994local}).

\begin{lemma}[Tur\`an--Nazarov Lemma]\label{lem:turan.nazarov} Let $T(x) := \sum_{k=0}^h \alpha_k e^{i m_k x}$ be a trigonometric polynomial with $\alpha_k\in \mathbb{C}$ and $m_1 < m_2 < \ldots < m_h \in \mathbb{Z}$. Then for any measurable set $E\subset [-\pi,\pi]$ we have 
\[
\sup_{x\in[-\pi,\pi]} |T(x)| \leq \left( \frac{C}{\mu(E)} \right)^{h-1} \sup_{x\in E} |T(x)|,
\]
where $C>0$ is an absolute constant.
\end{lemma}

Thus, we define 
\[
E_{\lambda}:= \LL{ \theta\in[-\pi,\pi] : \abs{G(re^{i\theta})} \leq \lambda \norm{G}_2 },
\]
Observe that $\norm{G}_2 \neq 0$ with probability one since $\Pr{\xi_0=0}=0$. Hence, from Tur\`an--Nazarov Lemma we have 
\[
\norm{G}_2 \leq \norm{ G }_{\infty}  \leq \PP{ \frac{C}{{\mu(E_{\lambda})}}}^{n} \sup_{\theta\in E_{\lambda}} \abs{G(re^{i\theta})} \leq \PP{ \frac{C}{{\mu(E_{\lambda})}}}^{n} \lambda \norm{G}_2,
\]
then
\begin{equation}\label{eqn:small01}
\mu(E_{\lambda}) \leq C_1 \lambda^{1/n}
\end{equation}
for some universal positive constant $C_1$. 

We assume that $\xi_j$ are fixed.
\begin{align*}
 \int_{-\pi}^{\pi} \abs{\log\abs{G(re^{i\theta})}}^4 \dd\mu(\theta) & = \int_{-\pi}^{\pi} \abs{\log\frac{\abs{G(re^{i\theta})}}{\norm{G}_2} + \log \norm{G}_2 }^4 \dd\mu(\theta) \\
 & \leq 8  \int_{-\pi}^{\pi} \abs{\log\frac{\abs{G(re^{i\theta})}}{\norm{G}_2}}^4 \dd\mu(\theta) + 8 \abs{\log\norm{G}_2}^4.
\end{align*}
We define the event $B:=\LL{ \abs{G(re^{i\theta})} > \norm{G}_2 }$, then
\begin{align*}
\int_{-\pi}^{\pi} \abs{\log\frac{\abs{G(re^{i\theta})}}{\norm{G}_2}}^4 \mathds{1}_B \dd\mu(\theta) & \leq \int_{-\pi}^{\pi} \PP{\log\frac{\sqrt{n+1} \norm{G}_2}{\norm{G}_2}}^4 \mathds{1}_B \dd\mu(\theta) \\
& \leq \frac{1}{2^4} \PP{\log(n+1)}^4.
\end{align*}
On the other hand, if $\mathds{P}_{\theta}$ indicates the distribution respect to $\mu(\theta)$, we have
\begin{align*}
\int_{-\pi}^{\pi} \abs{\log\frac{\abs{G(re^{i\theta})}}{\norm{G}_2}}^4 \mathds{1}_{B^c} \dd\mu(\theta) & = \int_{0}^{\infty} 4 u^3 \mathds{P}_{\theta}\PP{-\log\frac{\abs{G(re^{i\theta})}}{\norm{G}_2} \geq u, B^c } \dd u \\
& = \int_{0}^{\infty} 4 u^3 \mathds{P}_{\theta}\PP{ \abs{G(re^{i\theta})} \leq e^{-u}\norm{G}_2 } \dd u.
\end{align*}
Using \ref{eqn:small01}, we have
\[
\mathds{P}_{\theta}\PP{ \abs{G(re^{i\theta})} \leq e^{-u}\norm{G}_2 } \leq \min\LL{1,C_1 e^{-u/n}}.
\]
Thus,
\begin{align*}
& \int_{0}^{\infty} u^3 \mathds{P}_{\theta}\PP{ \abs{G(re^{i\theta})} \leq e^{-u}\norm{G}_2 } \dd u  = \int_{0}^{n\log n} u^3  \dd u + \int_{n\log n}^{\infty} u^3 \mathds{P}_{\theta}\PP{ \abs{G(re^{i\theta})} \leq e^{-u}\norm{G}_2 } \dd u \\
& \quad \leq  \frac{1}{4}(n\log n)^4 + C_1\int_{n\log n}^{\infty} u^3 e^{-u/n} \dd u \\
& \quad =  \frac{1}{4}(n\log n)^4 + C_1 n^4 \int_{\log n}^{\infty} v^3 e^{-v} \dd v \\
& \quad = \frac{1}{4}(n\log n)^4 + C_1 n^4 e^{-\log n}\PP{(\log n)^3 + 3(\log n)^2 + 6(\log n)+6} \\
& \quad \leq C_2\PP{ (n\log n)^4},
\end{align*}
for some universal positive constant $C_2$. Joining the above statements, we conclude
\[
 \int_{-\pi}^{\pi} \abs{\log\abs{G(re^{i\theta})}}^4 \dd\mu(\theta) \leq C_3 \PP{ (n\log n)^4 + \abs{\log\norm{G}_2}^4}.
\]
Now, we take the expectation respect to $\xi_j$,
\[
\E{  \int_{-\pi}^{\pi} \abs{\log\abs{G(re^{i\theta})}}^4 \dd\mu(\theta) } \leq C_3 \PP{ (n\log n)^4 + \E{\abs{\log\norm{G}_2}^4} },
\]
which means we need to analyze the expectation of $\abs{\log\norm{G}_2}^4$. For this, we observe 
\[
\abs{\log\norm{G}_2}^4  = \frac{1}{2^4} \abs{ \log\PP{\frac{1}{n}\norm{G}_2^2} + \log n}^4 \leq \frac{1}{2}\PP{(\log n)^4 + \CC{\log\PP{\frac{1}{n}\norm{G}_2^2}}^4 },
\]
and note 
\[
\frac{1}{n}\norm{G}_2^2 = \frac{1}{n} \sum_{j=0}^n \xi_j^2 r^{2j} \stackrel{n\to\infty}{\longrightarrow} \E{\xi_0^2}=1 \quad\mbox{a.s.},
\]
for all $r\in[1-n^{-11/10}, 1+n^{-11/10}]$. Moreover,
\begin{align*}
\E{\PP{\frac{1}{n}\norm{G}_2^2 - 1}^2} & = \frac{1}{n^2}\E{ \PP{\sum_{j=0}^n \CC{\xi_j^2 r^{2j} - 1}}^2} \\
& = \frac{1}{n^2}\PP{\sum_{j=0}^n \E{(\xi^2_j r^{2j} - 1)^2}  + 2\sum_{j_1<j_2} \E{(\xi_{j_1}^2 r^{2j_1} - 1)(\xi_{j_2}^2 r^{2j_2} - 1)} } \\
& \leq \frac{4e^4}{n^2}\PP{ n(\E{\xi_0^4}+1) } + \frac{2}{n^2}\sum_{j_1<j_2} (r^{2j_1}-1)(r^{2j_2}-1) \\
& \leq \frac{4e^4}{n^2}\PP{ n(\E{\xi_0^4}+1) } + \frac{2}{n^2}\sum_{j_1<j_2} (4e^4 j_1 n^{-11/10})(4e^4 j_2 n^{-11/10}) \\
& = \frac{4e^4}{n^2}\PP{ n(\E{\xi_0^4}+1) } + \frac{2^5 e^8}{n^2}\sum_{j_1<j_2}  j_1j_2 n^{-22/10} \\
& \leq \frac{4e^4}{n^2}\PP{ n(\E{\xi_0^4}+1) } + {2^5 e^8} n^{-2/10} \\
& \leq \frac{C_4}{n^{1/5}}, 
\end{align*}
for some positive constant $C_4$ which depends on $\E{\xi_0^4}$. 
We define $Y_n:=\frac{1}{n}\norm{G}_2^2$ and $D:=\LL{Y_n\in[0,1/2]}$, we have
\begin{align*}
& \E{ \abs{\log\PP{Y_n}}^4 } = \E{ \abs{\log\PP{Y_n}}^4 \mathds{1}_{D}} + \E{ \abs{\log\PP{Y_n}}^4 \mathds{1}_{D^c}}
\end{align*}
Thus, on $D$ is followed
\begin{align*}
& \E{ \abs{\log\PP{Y_n}}^4 \mathds{1}_{D}}  = \int_{0}^{\infty} 4u^3 \Pr{-\log Y_n \geq u, D} \dd u \\
& \quad \leq  \int_{0}^{\log n} 4u^3 \Pr{Y_n \leq \frac{1}{2} } \dd u + \int_{\log n}^{\infty} 4u^3 \Pr{Y_n \leq e^{-u}} \dd u \\
& \quad \leq (\log n)^4 \Pr{\abs{Y_n-1}\geq \frac{1}{2}} + \int_{\log n}^{\infty} 4u^3 \Pr{\xi_0^2 \leq n e^{-u}} \dd u \\
& \quad \leq C_4\frac{(\log n)^4}{n} + 4C n^{1/4} \int_{\log n}^{\infty} u^3 e^{-u/4} \dd u \\
& \quad \leq C_4\frac{(\log n)^4}{n} + C_5 \frac{(\log n)^3}{n^{1/4}} \leq C_6 \frac{(\log n)^4}{n^{1/4}}, 
\end{align*}
for some positive constants $C_5, C_6>0$ which depends on the distribution of $\xi_0$. For the left part, observing the inequality $\abs{\log(1+x)}^4 \leq 8 x^2$ for all $x \geq -\frac{1}{2}$, we have
\[
\E{ \abs{\log\PP{Y_n}}^4 \mathds{1}_{D^c}}  = \E{ \abs{\log\PP{1+(Y_n - 1)}}^4 \mathds{1}_{D^c}}  \leq 8\E{\abs{Y_n-1}^2} \leq \frac{8C_4}{ n^{1/5}}.
\]
We conclude that
\begin{align*}
\E{  \int_{-\pi}^{\pi} \abs{\log\abs{G(re^{i\theta})}}^4 \dd\mu(\theta) } & \leq C_7 \PP{ (n\log n)^4 + \frac{(\log n)^4}{n^{1/5}} }  \\
& \leq C_8 (n\log n)^4,
\end{align*}
for some positive constants $C_7, C_8$ which depend on the distribution of $\xi_0$.
\end{proof}

\section{Small ball probability} \label{sec:small.ball.probability}

In this section we prove Lemma \ref{lem:small.ball.probability}. We take $r\in [1-n^{-11/10},1+n^{-11/10}]$. In order to obtain a suitable bound for the expectation of the integral appearing in Lemma \ref{lem:expectation}, we need to control the corresponding small ball probability. Note that this integral is itself a random variable, as it depends on the random coefficients of the underlying Kac polynomial. We therefore consider the following statement.

\begin{lemma}\label{lem:small.ball.probability} Under the hypothesis of Theorem \ref{thm:main01}, there exist a constant $C_5>0$, which depends on the distribution of $\xi_0$, such that for all $a \in (0,1)$ and for all $r \in [1-n^{-11/10},1+n^{-11/10}]$ we have
\[
\int_{-\pi}^{\pi} \mathbb{P}\left( |G(re^{i\theta})| \leq a \right) \textnormal{d} \mu(\theta) \leq C_5 \PP{ \frac{1}{n^5} + \frac{1}{n^c} + n^{20c+6} a\log(a^{-1}) }
\]
for $c>1$.
\end{lemma}


\begin{proof}
%

From the Esseen's anti-cencentration inequality (see Theorem 7.17 in \cite{tao2006additive}), we note
\begin{equation}\label{eq:esseen}
\mathbb{P}\left( |P(re^{i\theta})| \leq a\right) \leq \mathbb{P}\left( |\Re P(re^{i\theta})| \leq a \right) \leq C_0 a\int_{-a^{-1}}^{a^{-1}} |g(t)| \dd t
\end{equation}
for some universal positive constant $C_0$ and $g$ is the characteristic function of  real part of the polynomial, $\Re P(re^{i\theta}) = \sum_{k=0}^{n} \xi_k r^k\cos(k\theta)$. If $g_k(t)$ is the characteristic function of $\xi_k r^k \cos(k\theta)$, we have 
\[
\abs{g_k(t)}^2 = \mathbb{E}_{\tilde{\xi}_k}\CC{ \cos\PP{ t r^k \cos(k\theta) \tilde{\xi}_k}},
\]
where $\tilde{\xi}_k$ is the symmetrization of $\xi_k$, i.e., $\tilde{\xi}_k:=\xi_k - \xi'_k$ and $\xi'_k$ is independent copy of $\xi_k$ (which also independent with all other considered random variables), and $ \mathbb{E}_{\tilde{\xi}_k}$ is the expectation respecto to $\tilde{\xi}_k$. From the inequality $x \leq \exp\PP{ -\frac{1}{2}(1-x^2) }$, which is valid for all $x\in\R$, is following that
\[
\abs{g_k(t)} \leq \exp\PP{ -\frac{1}{2}( 1 - \abs{g_k(t)}^2 ) }  = \exp\PP{ -\frac{1}{2}\CC{ 1 - \mathbb{E}_{\tilde{\xi}_k}\CC{ \cos\PP{ t r^k \cos(k\theta) \tilde{\xi}_k}} } }.
\]
From the above, the characteristic function $g(t)$ of $\Re P(re^{i\theta})$ satisfies
\begin{align}
\abs{g(t)} & = \prod_{k=0}^{n} \abs{g_k(t)} \leq \exp\PP{ -\frac{1}{2} \CC{ n - \sum_{k=0}^n \mathbb{E}_{\tilde{\xi}_k}\CC{ \cos\PP{ t r^k\cos(k\theta) \tilde{\xi}_k } } } } \nonumber \\
& \leq e^{-\frac{n}{2}} \exp\PP{ \frac{1}{2}\sum_{k=0}^n \mathbb{E}_{\tilde{\xi}_k}\CC{ \cos\PP{t r^k \cos(k\theta) \tilde{\xi}_k } } } \label{eqn:chafun}
\end{align}

Let $p:=\Pr{\tilde{\xi}_0=0}$, observe that $p<1$ since the distribution of $\xi_0$ is non--degenerate. From the previous expression \ref{eqn:chafun} and the definition of $\line{\xi}_k$ (see hypothesis $\mathcal{H}4$), using Fubini Theorem we have
\begin{align}
&  \int_{-\pi}^{\pi} \Pr{\abs{G_n(re^{i\theta})} \leq a } \dd \mu(\theta) \nonumber\\
& \quad \leq  a\int_{-a^{-1}}^{a^{-1}} e^{-n/2} \int_{-\pi}^{\pi} \exp\PP{\frac{1}{2}\sum_{k=0}^n \E[_{\tilde{\xi}_k}]{\cos\PP{ r^k\cos(k\theta)t\tilde{\xi}_k } 
\PP{\mathds{1}_{ \LL{\tilde{\xi}_k \neq 0} } + \mathds{1}_{ \LL{\tilde{\xi}_k = 0}} }
}}  \dd\mu(\theta) \dd t \nonumber \\
& \quad = a\int_{-a^{-1}}^{a^{-1}} e^{-(1-p)n/2} \int_{-\pi}^{\pi} \exp\PP{\frac{1-p}{2}\sum_{k=0}^n \E[_{\overline{\xi}_k }]{\cos(r^k\cos(k\theta)t\overline{\xi}_k) }}  \dd\mu(\theta) \dd t. \label{eqn:20may137}
 \end{align}
Fixing $t\in \CC{-a^{-1},a^{-1}}$, we define the following events:
 \begin{align*}
\Lambda_1(t,n) & := \LL{\theta\in[-\pi,\pi] : \sum_{k=0}^n \E[_{\overline{\xi}_k}]{\cos(t r^k\cos(k\theta)\overline{\xi}_k) } <\frac{n}{2} }, \\
\Lambda_2(t,n) & := [-\pi,\pi] \setminus \Lambda_1(t,n) 
\end{align*}
 
 The integral respect to $\dd \mu(\theta)$ in \ref{eqn:20may137} over $\Lambda_1(t,n)$ gives a exponential bound, $e^{-(1-p)n/2} e^{(1-p)n/4} = e^{-(1-p)n/4}$. Thus, we focus in the same integral but over $\Lambda_2(t,n)$.
 
 To handle the case of $\Lambda_2(t,n)$, we introduce the notation $\alpha_k(\theta):= r^k \cos(k\theta)$ and observe that 
 \begin{align*}
& \int_{-\pi}^{\pi} \PP{\frac{1}{2}\sum_{k=0}^n\E[_{\overline{\xi}_k}]{\cos(r^k\cos(k\theta)t\overline{\xi}_k) }}^{10} \dd \mu(\theta) \\
& \quad = \int_{-\pi}^{\pi} \PP{ \frac{1}{4} \sum_{k=0}^{n} \E[_{\overline{\xi}_k} ]{ e^{i\alpha_k(\theta)t\overline{\xi}_k}+e^{-i\alpha_k(\theta)t\overline{\xi}_k} } }^{10} \dd \mu(\theta) \\
& \quad = \int_{-\pi}^{\pi} \PP{ \frac{1}{4} \sum_{k=0}^{n} \PP{ \E[_{\overline{\xi}_k}]{ e^{i\alpha_k(\theta)t\overline{\xi}_k}} + \E[_{\overline{\xi_k}}]{e^{-i\alpha_k(\theta)t\overline{\xi}_k} } } }^{10} \dd \mu(\theta) \\
& \quad = \int_{-\pi}^{\pi} \PP{ \frac{1}{2} \sum_{k=0}^{n} \E[_{\overline{\xi}_k }]{ e^{i\alpha_k(\theta)t\overline{\xi}_k}} }^{10} \dd \mu(\theta) \\
& \quad = \frac{1}{2^{10}} \sum_{k_1,\ldots,k_{10}=0}^n \int_{-\pi}^{\pi} \E[_{\overline{\xi}_{k_{1:10}} }]{e^{it\sum_{l=1}^{10} a_{k_l}(\theta) \overline{\xi}_{k_l} }} \dd \mu(\theta)
\end{align*} 
due to that $\line{\xi}_k$ is a symmetric random variable and the fact \[ (\E{X})^2=\E{X}\E{X} = \E{X}\E{Y}=\E{XY} \] where $Y$ is a independent copy of $X$. Using Fubini Theorem and writing $\overline{\xi}_{k_{1:10}}:=\LL{\line{\xi}_{k_1},\ldots,\line{\xi}_{k_{10}}}$, we have 
\[
\frac{1}{2^{10}}  \sum_{k_1,\ldots,k_{10}=0}^n \int_{-\pi}^{\pi} \E[_{\overline{\xi}_{k_{1:10}} }]{e^{it\sum_{l=1}^{10} a_{k_l}(\theta) \overline{\xi}_{k_l} }} \dd \mu(\theta) =   \frac{1}{2^{10}} \sum_{k_1,\ldots,k_{10}=0}^n  \E[_{\overline{\xi}_{k_{1:10}} }]{\int_{-\pi}^{\pi} e^{it \sum_{l=1}^{10} a_{k_l}(\theta) \overline{\xi}_{k_l} } \dd \mu(\theta)}.
\]

Note that in the above sum there are $2^{10}$ terms where at least one $k_l$ is equal to zero in $\sum_{l=1}^{10} a_{k_l}(\theta) \overline{\xi}_{k_l} $ and the contribution of the sum of all these terms is at most $1$. Thus, we can assume for the following that all $k_l\geq 1$. To estimate the rest of main sum, we split the expectation as follows:
\begin{align*}
& \E[_{\overline{\xi}_{k_{1:10}} }]{\abs{\int_{-\pi}^{\pi} e^{it\sum_{l=1}^{10} a_{k_l}(\theta) \overline{\xi}_{k_l} } \dd \mu(\theta)}} \\
& \quad \leq \E[_{\overline{\xi}_{k_{1:10}} }]{\mathds{1}_{\LL{\max_{l\in\LL{1,\ldots,10}}\LL{\abs{\overline{\xi}_{k_l}}} \leq n^{-1}}}\cdot \abs{ \int_{-\pi}^{\pi} e^{it\sum_{l=1}^{10} a_{k_l}(\theta) \overline{\xi}_{k_l} } \dd \mu(\theta) } } \\ 
& \quad\quad + \E[_{\overline{\xi}_{k_{1:10}} }]{\mathds{1}_{\LL{\max_{l\in\LL{1,\ldots,10}}\LL{\abs{\overline{\xi}_{k_l}}} \geq n^{-1}} } \cdot \abs{ \int_{-\pi}^{\pi} e^{it\sum_{l=1}^{10} a_{k_l}(\theta) \overline{\xi}_{k_l} } \dd \mu(\theta) } }.
\end{align*} 
By hypothesis $\mathcal{H}4$, the first expectation of the right side satisfies 
\begin{align*}
& {\E[_{\overline{\xi}_{k_{1:10}} }]{\mathds{1}_{\LL{{\max_{l\in\LL{1,\ldots,10}}\LL{\abs{\overline{\xi}_{k_l}}} \leq n^{-1}}}} \cdot \abs{ \int_{-\pi}^{\pi} e^{it\sum_{l=1}^{10} a_{k_l}(\theta) \overline{\xi}_{k_l} } \dd \mu(\theta) } }} \\
& \quad \leq \E[_{\overline{\xi}_{k_{1:10}} }]{\mathds{1}_{\LL{\max_{l\in\LL{1,\ldots,10}}\LL{\abs{\overline{\xi}_{k_l}}} \leq n^{-1}} }}\\
& \quad = \Pr{\abs{\overline{\xi}_{k_l}} \leq n^{-1} \mbox{ for all $l=1,\ldots,10$}} \\
& \quad \leq \PP{C_2n^{-1/2}}^{10} = C_2^{10}n^{-5}.
\end{align*}

For the remaining expectation, we assume that $\max_{l\in\LL{1,\ldots,10}} \abs{\overline{\xi}_{k_l}} \geq n^{-1}$. Let $T(\theta) = \sum_{l=1}^{10} a_{k_l}(\theta) \line{\xi}_{k_l}$. Let $h=10$ and $c>0$ which its value will be defined later. Let \[\theta_1 < \theta_2<\ldots <\theta_L\in [-\pi,\pi]\] points in the unit circle which are solutions of \[ \abs{T'(\theta_j)} = n^{-h(c+1)},\quad j=1,\ldots,L.\]
Since the derivative $T'$ is a trigonometric polynomial of degree $\max_{l=1,\ldots,10} k_l \leq n$, it is followed that $L\leq 2n$. We divide the unit circle in $L$ intervals define as
\[
I_j := [\theta_j, \theta_{j+1}]\quad \mbox{ for }\quad j=1,\ldots, L-1 \quad \mbox{ and }\quad I_L := [-\pi,\pi] \setminus\PP{\cup_{j=1}^{L-1} I_j}.
\]
We say that $I_j$ is a good interval if $\abs{T'(\theta)} \geq n^{-h(c+1)}$ for all $\theta\in I_j$. We say that $I_j$ is bad if it is not good. We denote by $\mathcal{G}$ the set of good intervals and by $\mathcal{B}$ the ser of bad intervals. We state that if $I\in \mathcal{B}$, then $\mu(I) \leq n^{-(c+1)}$. If the opposite happens, then by construction there is an interval $I\subset [-\pi,\pi]$ such that
\[
\mu(I) > \frac{1}{n^{c+1}} \quad \mbox{ and } \quad \sup_{\theta\in I} \abs{T'(\theta)} \leq n^{- h(c+1)}.
\]
By the Turán's Lemma, we have
\[
\sup_{\theta\in[-\pi,\pi]} \abs{T'(\theta)} \leq \PP{\frac{C}{\mu(I)}}^{h-1} n^{-h(c+1)} \leq C^{h-1} n^{-(c+1)}.
\]
On the other side, there a universal positive constante $C_1$ such that
\begin{align*}
\int_{-\pi}^{\pi} \abs{T'(\theta)}^2 \dd \theta& = \int_{-\pi}^{\pi} \PP{ -\sum_{l=1}^{h} r^{k_l} k_l \sin(k_l\theta) \overline{\xi}_{k_l} }^2 \dd \mu(\theta) \\
& = \int_{-\pi}^{\pi} \PP{ \sum_{l_1=1}^{h}\sum_{l_2=1}^{h} k_{l_1}k_{l_2} r^{k_{l_1}+k_{l_2}} \overline{\xi}_{k_{l_1}}\overline{\xi}_{k_{l_2}} \sin(k_{l_1}\theta)\sin(k_{l_2}\theta) } \dd \mu(\theta) \\
& = \sum_{l=1}^{h} k_{l}^2 r^{2k_{l}} (\overline{\xi}_{k_{l}})^2 \int_{-\pi}^{\pi} \sin^2(k_{l}\theta) \dd \mu(\theta) \\
& \geq \frac{C_1}{2} \sum_{l=1}^{h} (\overline{\xi}_{k_{l}})^2 \geq \frac{C_1}{2} \max_{l}\LL{(\overline{\xi}_{k_{l}})^2} \geq \frac{C_1}{2} n^{-2},
\end{align*}
which means that if $c>1$ the inequality $\CC{C^{h-1}n^{-(c+1)}}^2 \geq \frac{C_1}{2}n^{-2}$ does not satisfy for all sufficiently large $n$ since $h=10$. Thus, for all sufficiently larger $n$ we have
\[
\mu\PP{\cup_{I\in\mathcal{B}} I} \leq \frac{2n}{n^{c+1}} = \frac{2}{n^{c}}
\]
with $c>1$.

Finally, we note using the above observations and the integration by parts,
\begin{align*}
\abs{ \int_{-\pi}^\pi  e^{itT(\theta)} \dd \mu(\theta) } & = \abs{ \int_{-\pi}^\pi  \PP{ \mathds{1}_{ \cup_{I\in\mathcal{B}}I }(\theta) + \mathds{1}_{ \cup_{I\in\mathcal{G}}I }(\theta) }  e^{itT(\theta)} \dd \mu(\theta) } \\
& \leq \frac{2}{n^c} + \sum_{I \in \mathcal{G}} \abs{ \int_{I} e^{itT(\theta)} \dd \mu(\theta) } \\
& \leq \frac{2}{n^c} + \frac{2n \cdot n^{h(c+1)}}{\abs{t}} + \sum_{I \in \mathcal{G}} \frac{1}{\abs{t}} \abs{ \int_{I} \frac{T''(\theta)}{(T'(\theta))^{2}} \dd \mu(\theta)} \\
& \leq \frac{2}{n^c} + \frac{2n \cdot n^{h(c+1)}}{\abs{t}} + \frac{2n}{\abs{t}} \abs{ \int_{I} \frac{T''(\theta)}{(T'(\theta))^{2}} \dd \mu(\theta)} \\
& \leq \frac{2}{n^c} + \frac{2n \cdot n^{h(c+1)}}{\abs{t}} + \frac{2n \cdot n^{2h(c+1)}}{\abs{t}} { \int_{I} \abs{T''(\theta)} \dd \mu(\theta)} \\
& \leq \frac{2}{n^c} + \frac{2n \cdot n^{h(c+1)}}{\abs{t}} + \frac{20 n \cdot n^{2h(c+1) + 2}}{\abs{t}} \sum_{l=1}^{10} |\overline{\xi}_l | . \\
\end{align*}

Now, taking the expectation the above expression, we observe that there a positive constant $C_2$ depending on the distribution of $\xi_0$ such that
\begin{align*}
& \E[_{\overline{\xi}_{k_{1:10}} }]{\int_{-\pi}^{\pi} e^{it\sum_{l=1}^{10} a_{k_l}(\theta) \overline{\xi}_{k_l} } \dd \mu(\theta)} \\
& \quad \leq \frac{C^{10}}{n^{5}} + \frac{4}{n^c} + \frac{4n \cdot n^{h(c+1)}}{\abs{t}} + \frac{40n \cdot n^{2h(c+1) + 2}}{\abs{t}} \\
& \quad \leq \frac{C^{10}}{n^{5}} + \frac{4}{n^c} + \frac{4 n^{h(c+1)+1}}{\abs{t}} + \frac{40 n^{2h(c+1) + 3}}{\abs{t}} \\
& \quad \leq C_2 \PP{\frac{1}{n^{5}} + \frac{1}{n^c} + \frac{n^{2h(c+1) + 3}}{\abs{t}}} .    
\end{align*} 
where $c>1$. Hence there exista a positive constante $C_3$ depending on the distribution of $\xi_0$ such that
\begin{align*}
& \int_{-\pi}^{\pi} \PP{\frac{1}{2}\sum_{k=0}^n\E[_{\overline{\xi_k}}]{\cos(r^k\cos(k\theta)t\overline{\xi}_k) }}^{10} \dd \mu(\theta) \\
&\quad\quad = \frac{1}{2^{10}} \sum_{k_1,\ldots,k_{10}=0}^n \int_{-\pi}^{\pi} \E[_{\overline{\xi}_{k_{1:10}} }]{e^{it\sum_{l=1}^{10} a_{k_l}(\theta) \overline{\xi}_{k_l} }} \dd \mu(\theta) \\
& \quad\quad  \leq 1 + C_2 n^{10} \PP{\frac{1}{n^{5}} + \frac{1}{n^c} + \frac{  n^{2h(c+1) + 3}}{\abs{t}}} \\
& \quad\quad  \leq C_3n^{10} \PP{\frac{1}{n^5} + \frac{1}{n^c} + \frac{ n^{2h(c+1) + 3}}{\abs{t}}}.
\end{align*} 

By Markov's inequality, we have that there exists a positive constant $C_4$ depending on the distribution of $\xi_0$ such that
\begin{align*}
\mu(\Lambda_2(t,n)) & \leq \frac{2^{10}}{(1-p)^{10} n^{10}} \CC{ C_3 n^{10} \PP{\frac{1}{n^5} + \frac{1}{n^c} + \frac{ n^{2h(c+1) + 3}}{\abs{t}}} } \\
& \leq C_4 \PP{\frac{1}{n^5} + \frac{1}{n^c} + \frac{ n^{2h(c+1) + 3}}{\abs{t}}}.
\end{align*}

From expression \ref{eqn:20may137} and previous statements, for $a<1$ we observed
\begin{align*}
 \int_{-\pi}^{\pi} \Pr{\abs{G_n(re^{i\theta})} \leq a } \dd \mu(\theta) & \leq 2e^{-(1-p)n/4} + a\int_{-a^{-1}}^{a^{-1}} \mu(\Lambda_2(t,n)) \dd t \\
 & \leq 2e^{-(1-p)n/4} + 2a \int_{0}^{a^{-1}} \mu(\Lambda_2(t,n)) \dd t \\
 & \leq 2e^{-(1-p)n/4} + 2a + 2a\int_{1}^{a^{-1}} \mu(\Lambda_2(t,n)) \dd t \\
 & \leq 2e^{-(1-p)n/4} + 2a + 2C_4\PP{ \frac{1}{n^5} + \frac{1}{n^c} + n^{2hc+6} a\log(a^{-1}) } \\
 & \leq C_5 \PP{ \frac{1}{n^5} + \frac{1}{n^c} + n^{2hc+6} a\log(a^{-1}) },
\end{align*}
for some positive constant $C_5$ depending on the distribution of $\xi_0$.

\end{proof}

\bibliographystyle{amsplain}
\bibliography{rootDisc}

\end{document}